\documentclass[reqno,a4paper,12pt]{amsart}
\usepackage[top=1.25in,bottom=1.25in,left=1in,right=1in]{geometry}

\usepackage{amsmath,amssymb,amsfonts}
\usepackage{amsthm}
\usepackage{xpatch}
\usepackage{cases}
\usepackage{interval}
\usepackage{enumitem}
\usepackage{graphicx}
\usepackage{url}
\usepackage[colorinlistoftodos]{todonotes}
\usepackage{xcolor}
\usepackage{tikz-cd}
\usepackage{booktabs}
\usepackage[utf8]{inputenc}
\usepackage[english]{babel}

\usepackage{orcidlink}

\usepackage{doi}

\usepackage{soul}

\usepackage{bookmark}
\usepackage{hyperref}
\hypersetup{
   colorlinks=true,
   linkcolor=blue,
   filecolor=magenta,      
   urlcolor=blue,
   citecolor=[rgb]{0.4, 0, 0.3},
}
\makeatletter
 
  \@addtoreset{equation}{section}
\makeatother

\xpatchcmd{\proof}
  {#1}
  {\itshape\textbf{#1}}
  {}
  {}

\newtheorem{thm}{Theorem}[section]
\newtheorem{lem}[thm]{Lemma}

\newtheorem{cor}[thm]{Corollary}

\newtheorem*{theorem*}{Theorem}

\theoremstyle{definition}
\newtheorem{dfn}[thm]{Definition}
\newtheorem{exam}[thm]{Example}
\newtheorem{rem}[thm]{Remark}

\newcommand{\ZZ}{\ensuremath{\mathbb{Z}}}

\begin{document}
\title{Lifting of cycles in functional graphs}

\author[T. Nara]{Tadahisa Nara}
\address{Faculty of Information Design, Tokyo Information Design Professional University, 2-7-1 Komatsugawa, Edogawa-ku, Tokyo, 132-0034, Japan} \email{nara@tid.ac.jp}

\begin{abstract}
For a given function from a set to itself, we can define a directed graph called the functional graph, 
where the vertices are the elements of the set, 
and the edges are all the pairs of inputs and outputs for the function. 
In this article we consider functional graphs on $\ZZ/ m\ZZ$ with respect to polynomial functions.
The main result describes the behavior of cycles in functional graphs on $\ZZ/p^n \ZZ$ 
while $n$ is increasing, where $p$ is a prime number.
\end{abstract}



\maketitle

\section{Introduction}

For a given function from a set to itself, we can define a directed graph called the functional graph, 
where the vertices are the elements of the set, 
and the edges are all the pairs of inputs and outputs for the function.
In recent years functional graphs on finite sets are actively studied, 
in particular concerning polynomial functions or rational functions.
On finite fields many remarkable results are given  
from multiple perspectives. 
For example,
\cite{konyagin_functional_2016, bach_number_2013} consider the number of non-isomorphic functional graphs 
and give bounds for that. 
\cite{%
flynn_graph_2014, mans_functional_2019} consider structural aspect of functional graphs, such as cycles or connected components.

In this article we consider functional graphs on the finite rings $\ZZ / m\ZZ$ with respect to polynomial functions.
As shown in Theorem \ref{thm:graph-isom} the structure of a functional graph is determined only from 
that on the primary components of $\ZZ / m\ZZ$.
However, the rules that govern the latter is not clear, and that is why we are interested in functional graphs on $\ZZ/p^n \ZZ$.
The main result is a description on behavior of cycles in functional graphs on $\ZZ/p^n \ZZ$ 
while $n$ is increasing (Theorem \ref{thm:main-0}). 
In that study we can find some sort of differential information on the function is a key for classifying the behavior.
In \cite{yoshioka_properties_2018} a similar problem is considered, focusing on 
the size of cycles and the expression of the vertices, but limited to the case for Chebyshev polynomials from a cryptographic perspective.

As a related study,  
counting the number of functions realized as polynomial functions, which corresponds to the number of functional graphs given by polynomials, 
is a classical problem 
(\cite{singmaster_polynomial_1974, keller_counting_1968}). 
In \cite{specker_ring_2023} the ring of functions realized as polynomial functions was investigated and the structure was described. 
In \cite{chen_polynomial_1995} a kind of polynomial representable functions 
from $\{0,1,..., m-1\}$ to $\{0,1,..., n-1\}$ was proposed for $m\neq n$, 
and it was studied in more sophisticated formulation and generalized to residue class rings of Dedekind domains in \cite{li_polynomial_2019}.

Alternatively, permutation polynomials, that is, polynomials which are bijective as functions,
are widely studied in this area (\cite{singh_method_2021, gorcsos_permutation_2018}).

\section{Preliminaries}
Throughout the paper, we denote the ring of integers modulo $m$ by $\ZZ_m:=\ZZ/ m\ZZ$ and $p$ is any prime number unless otherwise mentioned. 
Let $f: \ZZ \to \ZZ$ be a polynomial function, and then we can naturally regard $f$ as a function from $\ZZ_m$ to $\ZZ_m$ for any $m$. 
So by abusing a notation we use the same symbols for such functions 
for different domains.
We denote a functional graph of a polynomial function $f$ on $\ZZ_m$ by
\[
G(f, \ZZ_m) := (V, E),
\]
where vertices $V=\ZZ_m$ and the directed edges $E=\{(v, f(v)): v \in V\}$.
Also in this article the notation $f^i$ means the $i$-th iteration of a function $f$.

\begin{rem}
A functional graph must have outdegree $1$ for any vertex, and so 
it is represented by a set of edges in the form 
\[
\{(0, b_0), (1,b_1)),..., (m-1,b_{m-1})\},\quad b_i \in \ZZ_m.
\]
On the other hand, not all the graphs in the form are given by polynomial functions 
unless $m$ is a prime.
Concerning the number of functional graphs actually given by polynomial functions, see 
\cite{singmaster_polynomial_1974, keller_counting_1968, specker_ring_2023},
and also \cite{bulyovszky_polynomial_2017}, which considers conditions whether a given function can be realized as a polynomial function.

\end{rem}

\section{Fundamental Results}

The following theorem is a fundamental fact
for our situation,
which says the structure of a functional graph on $\ZZ_m$ is 
determined only from that on the primary components of $\ZZ_m$.

\begin{thm}\label{thm:graph-isom}
Let $m,n$ be coprime integers and $f: \ZZ \to \ZZ$ a polynomial function.
Then we have a graph isomorphism 
\begin{align*}
G(f, \ZZ_m) \otimes G(f, \ZZ_n) \simeq G(f, \ZZ_{mn}),
\end{align*}
where the left-hand side is a tensor product of graphs.
\end{thm}

\begin{proof}
We claim that the isomorphism of the Chinese remainder theorem (CRT) induces the graph isomorphism.
Indeed, let $\phi$ be the isomorphism 
\begin{align}
\phi: \ZZ_m \times \ZZ_n \to \ZZ_{mn}
\end{align}
induced by $(x, y) \mapsto bnx+amy$ for $x,y\in\ZZ$, 
where $a,b\in\ZZ$ are such that $am+bn=1$.
Then this naturally induces the map
\[
\phi_G: G(f, \ZZ_m) \otimes G(f, \ZZ_n) \to G(f, \ZZ_{mn}).
\]

By definition $((x\pmod m,y \pmod n),(x'\pmod m,y'\pmod n))$ is a directed edge of $G(f, \ZZ_m) \otimes G(f, \ZZ_n)$ 
if and only if $(x\pmod m, x'\pmod m)$ and $(y\pmod n, y'\pmod n)$ are directed edges of  
$G(f, \ZZ_m)$ and $G(f, \ZZ_n)$, respectively.
So, to show $\phi_G$ is an isomorphism, we have to prove the equivalence
\begin{align*}
&f(x) =x' \pmod m,\ f(y)=y' \pmod n\\
&\qquad \Longleftrightarrow 
f(bnx+amy) =bnx'+amy' \pmod {mn}.
\end{align*}
Note, because of $\gcd(m,n)=1$, the RHS is 
equivalent to 
\begin{align*}
f(bnx+amy) =bnx'+amy'
\end{align*}
both modulo $m$ and $n$. 
Since $f$ is a polynomial, we have
\begin{align*}
f(bnx+amy) &= f((1-am)x+amy)=f(x) \pmod m,\\
bnx'+amy' &= (1-am)x'+amy'=x' \pmod m 
\end{align*}
and similarly we have
\[
f(bnx+amy)=f(y) \pmod n,\ 
bnx'+amy'=y' \pmod n, 
\]
which implies the desired equivalence.
\end{proof}

From now on we frequently use the notation
\[
C = \{v_0, v_1, \dots, v_{k-1}\}
\]
for a cycle, which precisely means the directed graph $(V, E)$, where
\[
V =\{v_0, v_1, \dots, v_{k-1}\}, \quad E = \{(v_0, v_1), (v_1, v_2),\dots, (v_{k-1}, v_0)\}.
\]

\begin{cor}
If there is a cycle of size $k$ in $G(f, \ZZ_m)$ and a cycle of size $l$ in $G(f, \ZZ_n)$ with $\gcd(m,n)=1$, 
then there is a cycle of size $\operatorname{lcm}(k,l)$ in $G(f, \ZZ_{mn})$.
\end{cor}
\begin{proof}
Let $C_m=\{v_0,\dots, v_{k-1}\}$ and $C_n=\{u_0,\dots, u_{l-1}\}$ be cycles in $G(f,\ZZ_m)$ and $G(f,\ZZ_n)$, respectively.
Then for any $i,j$ there are directed edges
\[
((v_i, u_j), (v_{i+1}, u_{j+1}))
\]
in $G(f,\ZZ_m) \otimes G(f, \ZZ_n)$, making a cycle, where we take $i,j$ modulo $k,l$, respectively.
Note that the least positive integer $d$ such that  
$i=i+d \pmod k$ and $j=j+d \pmod l$ is the size of the cycle, 
which is by definition $d=\operatorname{lcm}(k,l)$.
Then by the isomorphism of the theorem there exists a cycle of size $d$ in $G(f, \ZZ_{mn}) $.
\end{proof}

\section{Main Results}

Next we would determine graphs $G(f, \ZZ_{p^n})$ from the information on $G(f, \ZZ_p)$ 
for $p$ prime.
But we soon find it fails as expected, that is, 
$G(f, \ZZ_p) \simeq G(g, \ZZ_p)$ does not mean $G(f, \ZZ_{p^n}) \simeq G(g, \ZZ_{p^n})$.
Therefore we suggest a problem that 
how to obtain the structural behavior of functional graphs $G(f, \ZZ_{p^n})$ for $n$
from the information on $G(f, \ZZ_p)$ and some properties of $f$.
In particular, we focus on the cycles in the graphs.

Throughout this article we use the terms \textit{lifted graph} and \textit{lifted cycle} as follows.
Let $\pi: \ZZ_{p^{n}} \to \ZZ_{p^m}$ for any $n>m$, be the natural projection.
Then we can naturally define, for a fixed $f$, 
\begin{align*}
\pi_G: G(f, \ZZ_{p^n}) \to G(f, \ZZ_{p^m})
\end{align*}
by $v \mapsto \pi(v)$ for vertices and $(u, v) \mapsto (\pi(u), \pi(v))$ for edges.
Let $S=(V,E)$ be a subgraph of $G(f, \ZZ_{p^m})$, and then we call 
the subgraph of $G(f, \ZZ_{p^n})$ induced by $\pi^{-1}(V)$, the lifted graph of $S$.
A lifted cycle of $S$ means a cycle in the lifted graph.

\begin{exam}
Let $p=3$, $f=x^3+2$. Then $C =\{\bar{0}, \bar{1}, \bar{2}\}$ is a cycle in $G(f, \ZZ_3)$.
The lifted cycle in $G(f, \ZZ_{9})$ of $C$ is the cycle $\{\bar{1}, \bar{3}, \bar{2}\}$, and 
the lifted graph in $G(f, \ZZ_{9})$ of $C$ is the whole graph $G(f, \ZZ_{9})$.
\end{exam}

Before the main result we introduce a key concept \textit{multiplier}. 
As a reference we have \cite{silverman_arithmetic_2007}.
Usually this concept is used to study functions and the periodic points in the context of dynamical systems over infinite fields.
So our result is an answer to how the concept is involved in the context of finite rings.

\begin{dfn}
Let $C =\{v_0, v_1, \dots, v_{k-1}\}$ be a cycle of size $k$ in a functional graph $G(f, \ZZ_{p^n})$ for a polynomial function $f$.
Then the \textbf{multiplier} of $C$ is defined as 
\[
\lambda(C) := (f^{k})'(v_0) = \prod_{v \in C} f'(v),
\]
where $f^k$ is the $k$-th iteration of $f$ and $(f^k)', f'$ denote the derivatives of the functions.
The second equality is given by the chain rule and this allows us to  
choose any vertex of $C$ instead of $v_0$ in the second term.
In most cases we use only the image of $\lambda(C)$ in $\ZZ_p$, denoted by $\bar\lambda(C)$.
\end{dfn}

\begin{thm}\label{thm:main-0}
Let $C$ be a cycle of size $k$ in $G(f, \ZZ_{p^n})$ for a polynomial function $f$ with multiplier $\lambda=\lambda(C)$.
Put $\bar{\lambda}$ as the image of $\lambda$ in $\ZZ_p$.
In what follows we always mean the lifted graph of $C$ as in $G(f, \ZZ_{p^{n+1}})$.
\begin{enumerate}[label={\upshape(\arabic*)}]
\item
If $\bar{\lambda} = 0$, then only one cycle  
is in the lifted graph of $C$, which has size $k$. 
\item
If $\bar{\lambda} = 1$, then 
the lifted graph of $C$ consists of only one cycle of size $kp$, or $p$ cycles of size $k$.
The latter case occurs if and only if $r_v=0$ for a vertex $v$ of $C$, where $r_v$ is defined in Definition \ref{def:r} below.
\item
If $\bar{\lambda} \neq 0, 1$, 
then 
the lifted graph of $C$ consists of 
a cycle of size $k$ and $(p-1)/m$ cycles of size $mk$, 
where $m=m(\bar{\lambda})$ is the multiplicative order of $\bar{\lambda}$, 
i.e. the least positive integer s.t. $\bar{\lambda}^m = 1$.
\end{enumerate}
\end{thm}

\begin{rem}
For (2), we can choose any vertex $v$ of $C$ for $r_v$ (see Lemma \ref{lem:property-r} (2)).
\end{rem}

From now on we use the following notation for integer ranges:
\[
[a, b) := \{a, a+1, \dots, b-1\}.
\]

\begin{dfn}\label{def:r}
For a vertex $v$ of a cycle $C$ of size $k$, in a functional graph $G(f, \ZZ_{p^n})$, 
define $r_v \in [0, p)$ by 
\begin{align*}
r_v := \frac{f^k(a_v) - a_v}{p^n}\, \%\, p,
\end{align*}
where $a_v$ is the unique representative in $[0,p^n)$ of $v$, and 
$A \mathbin{\%} B \in [0, B)$ denotes the remainder for division of $A$ by $B>0$.
\end{dfn}

In other words, $r_v \in [0, p)$ is the integer such that
\[
f^k(a_v) = a_v + r_v p^n \pmod{p^{n+1}}.
\]

\begin{rem}
As we will see in Lemma \ref{lem:property-r} (1), 
if $\bar{\lambda}(C)=1$, then in the definition of $r_v$, 
we can use any representative of $v$ instead of $a_v$.
But the quantity $r_v$ is used in the proof of Theorem \ref{thm:main-0}, where the assumption does not necessarily hold, 
and so we define it using $a_v$.

Also note that $n$ is dependent on $v$.
\end{rem}

The next lemma is an easy fact used in the proof of the theorem.
\begin{lem} \label{lem:size-multiple}
If $C$ is a cycle of size $k$ in $G(f,\ZZ_{p^n})$,
then the size of any lifted cycle of $C$ 
is a multiple of $k$.
\end{lem}
\begin{proof}
Assume the size of a lifted cycle of $C$ is $l$ and write $l=Ak+B$ 
for some $A\geq 0$ with $0\leq B<k$.
Then for a vertex $v'$ in the lifted cycle, $v' = f^l(v')$, in particular $v=f^l(v)$, 
where $v\in \ZZ_{p^n}$ is the image of $v'$.
Since $f^l(v) = f^{Ak+B}(v) = f^B(f^{Ak}(v)) = f^B(v)$, it requires $B=0$.
\end{proof}

\begin{proof}[Proof of Theorem \ref{thm:main-0}]
Let $v\in \ZZ_{p^n}$ be a vertex of $C$, 
and $a \in\ZZ$ any representative of $v$.
Then
\begin{align*}
a = a_v + b p^n \pmod{p^{n+1}}
\end{align*}
with some $a_v \in [0, p^n)$ and $b \in [0,p)$.
Put $g:=f^k$ and then the Taylor expansion is
\begin{align*}
g(x_0 + h) = g(x_0) + g'(x_0)h + \frac{g''(x_0)}{2!}h^2+\cdots,
\end{align*}
where note the coefficients $\frac{g^{(n)}(x_0)}{n!}$ are integers and the series is actually finite since $g$ is a polynomial over $\ZZ$.
So we have for $n \geq 1$
\begin{align}
g(a) 
&= g(a_v + b p^n) \pmod{p^{n+1}} \nonumber\\
&= g(a_v) + g'(a_v) (b p^n) +  \frac{g''(a_v)}{2!} (b p^n)^2 +\cdots \nonumber\\
&=a_v+(r_v+\gamma b)p^n,
\end{align}
where 
$r_v$ is defined in Definition \ref{def:r} and $\gamma := g'(a_v)$.
Then the iteration gives
\begin{align}
g^2(a)
&=a_v+(r_v+\gamma(r_v+\gamma b))p^n \pmod {p^{n+1}} \nonumber\\
&=a_v+(r_v+r_v\gamma+\gamma^2 b)p^n, \nonumber\\
&{\hspace{80pt} \vdots} \nonumber\\
g^N(a) 
&= a_v+ (r_v\sum_{l=0}^{N-1} \gamma^l +\gamma^N b)p^n \pmod {p^{n+1}}. \label{eq:iter_n}
\end{align}

If $\bar{\lambda}=0$, i.e. $\gamma=g'(a_v)=0 \pmod p$, then $g(a)=a_v + r_v p^n \pmod {p^{n+1}}$, which is independent of $b$.
This means  all the vertices in the form $a_v+cp^n$  in the lifted graph 
are mapped to a single vertex by $f^k$. 
In particular, $f^k(a_v+r_v p^n)=a_v+r_v p^n \pmod {p^{n+1}}$.
Now write $C=\{v_0, v_1, \dots, v_{k-1}\}$ 
and then
we claim the only lifted cycle of $C$ is
\[
C' = \{a_0+r_0 p^n, a_1+r_1p^n,..., a_{k-1}+r_{k-1}p^n\} \pmod{p^{n+1}},
\]
where $a_j\in [0, p^n)$ is the representatives of $v_j$ and $r_j := r_{v_j}$.
Indeed, as seen above 
we have a cycle
\begin{align*}
\{a_0+r_0p^n, f(a_0+r_0p^n),..., f^{k-1}(a_0+r_0 p^n)\} \pmod{p^{n+1}}
\end{align*}
in $G(f, \ZZ_{p^{n+1}})$ and this is actually identical to $C'$ 
since for any $j$, 
\begin{align*}
f^j(a_0+r_0 p^n) 
&= f^{j+k}(a_0+r_0 p^n) \pmod{p^{n+1}}\\
&= f^k(f^j(a_0+r_0 p^n))\\
&= f^k(a_{j}+r' p^n)\\
&= a_{j} + r_{j} p^n,
\end{align*}
where $r'$ is some unknown integer but make 
no differences in the result.
Now we have to show there is no other cycle than $C'$.
Let $v$ be a vertex in the lifted graph of $C$ but not in $C'$. 
For a representative of $v$ we can take $a_j + s p^n$ with $s \neq r_j$ for some $j$. 
For $i <k$ clearly $f^i(a_j + s p^n)\neq a_j + s p^n$ modulo ${p^{n+1}}$ since  it holds modulo $p^n$.
Then as we have seen above, independent of $s$,
\[
g(a_j + s p^n) = f^k(a_j + s p^n) = a_j+r_j p^n,
\]
which is in $C'$ and this never goes outside $C'$ by iteration of $f$, 
which implies $C'$ is the only lifted cycle of $C$.

If $\bar{\lambda}=1$, then by \eqref{eq:iter_n} 
\[
g^N(a)=a_v+(Nr_v+b)p^n \pmod {p^{n+1}}.
\]
This means the smallest $N>0$ s.t. $g^N(a)=a \pmod {p^{n+1}}$ is 
\begin{align*}
N = 
\begin{cases}
1 & (r_v=0),\\
p & (r_v\neq 0),
\end{cases}
\end{align*}
thus by Lemma \ref{lem:size-multiple} 
the size of the cycles in the lifted graph is 
\begin{align*}
\begin{cases}
k & (r_v=0),\\
kp & (r_v\neq 0).
\end{cases}
\end{align*}

Finally, assume $\bar{\lambda}\neq 0, 1$.
Now if $N$ equals the multiplicative order of $\bar{\lambda}$, denoted by $m(\bar{\lambda})$, 
then since 
\[
\sum_{l=0}^{N-1} \gamma^l = \frac{\gamma^N-1}{\gamma-1}=0 \pmod p, 
\]
we have $g^N(a)=a_v+b p^n = a\pmod{p^{n+1}}$ by \eqref{eq:iter_n},
which means the period of $a \pmod{p^{n+1}}$ for $g$ is at most $m(\bar{\lambda})$.
If it is strictly smaller, that is, $N < m(\bar{\lambda})$ and $g^N(a)=a_v+b p^n \pmod{p^{n+1}}$,
then 
\begin{align*}
&r_v \frac{\gamma^N-1}{\gamma-1}+\gamma^N b = b \pmod{p}\\
&\Longleftrightarrow r_v = b\, (1-\gamma) \pmod p, 
\end{align*}
which gives
\[
g(a)=a_v + b p^n \pmod{p^{n+1}}.
\]
So we have proved that 
the smallest $N>0$ s.t. $g^N(a)=a \pmod {p^{n+1}}$ is
\begin{align*}
N = 
\begin{cases}
1 & (b=r_v(1-\gamma)^{-1} \pmod p),\\
m(\bar{\lambda}) & (b: \text{otherwise}).
\end{cases}
\end{align*}
So the iteration of $f$ with $a_v + l p^n \pmod {p^{n+1}}$ generates 
a cycle of size $m(\bar{\lambda})\cdot k$ for each $l \pmod p$ except one value, in which case the size is $k$. 
Note $km$ distinct vertices are necessary to make a cycle of size $km$, 
thus the number of cycles of size $km$ is $(kp-Tk)/km=(p-1)/m$ since the number of vertices in the lifted graph of $C$ is $kp$.
\end{proof}

\begin{rem}
In most cases after several lifting of a cycle with multiplier $1$, 
eventually we attain $r_v \neq 0$, and so by Theorem \ref{thm:main-0} (2) 
the size of cycles is expanding by a factor of $p$. 
But there is an example such that $r_v$ is staying in $0$ while successive liftings.
Let
\[
f(x) =3x-x^3,
\]
and $C_i:=\{\bar{-2}, \bar2\}$ in $G(f, \ZZ_{p^i})$ for $i=1,2,3,\cdots$ is a sequence of lifted cycles,
which is derived from ``global'' cycle $\{-2, 2\}$ in $G(f, \ZZ)$.
Now, for example, if $p=5$, then we have  
\[
\bar\lambda(C_i)= f'(-2)f'(2) = 81 = 1 \pmod 5.
\]
Then recall $r_v$ is independent of choice of representative of $v$ as mentioned above. 
So for $v=\bar{2}$ in $C_i$ for every $i$, 
\begin{align*}
r_{v} 
= \frac{f^2(2)-(2)}{p^i} \% p
=0.
\end{align*}
\end{rem}

As a complement we give a lemma about the behavior of the multiplier, which is used in Corollary \ref{cor:main}.
\begin{lem} \label{lem:property-multiplier}
Let $C$ be a cycle in a functional graph $G(f, \ZZ_{p^n})$ and $C'$ a lifted cycle in $G(f, \ZZ_{p^{n+1}})$ of $C$.
Let the size of $C$ be $k$.
If the size of $C'$ equals $k$, then $\bar\lambda(C')=\bar\lambda(C)$.
Otherwise, $\bar\lambda(C')=1$.
\end{lem}
\begin{proof}
We have the following commutative diagram:
\begin{tikzcd}
\ZZ_{p^{n+1}} \arrow{r}{\pi_{n+1,n}}  \arrow{rd}{\pi_{n+1}} & \ZZ_{p^n} \arrow{d}{\pi_n} \\
& \ZZ_p.
\end{tikzcd}

Write $C= \{v_0, v_1, \dots, v_{k-1}\}$ and $g:=f^k$.
If the size of $C'$ equals $k$, then we can write 
$C'=\{v_0', v_1', \dots, v_{k-1}'\}$ with some $v_i'$ such that $\pi_{n+1, n}(v_i') = v_i$.
Since $g'$ is also a polynomial function, we have
\[
\bar\lambda(C')=\pi_{n+1}(\prod g'(v_i')) = \pi_n(\prod g'(\pi_{n+1, n}(v_i'))) = \pi_n(\prod g'(v_i)) = \bar\lambda(C).
\]

Otherwise, the size of $C'$ equals $kp$ or $km$ by Theorem \ref{thm:main-0}, 
where $m$ is the multiplicative order of $\bar\lambda(C)$.
Now note in general if the size of $C'$ equals $kM$ for some integer $M$, 
then $C'$ must contain $M$ distinct vertices which are mapped to $v_i$ by $\pi_{n+1, n}$ for each $i$. 
So by similar computation as above we have 
\[
\bar\lambda(C')=\bar\lambda(C)^M.
\]
In the first case $\bar\lambda(C)=1$ by Theorem \ref{thm:main-0} and so $\bar\lambda(C')=1$.
In the second case $\bar\lambda(C')=\bar\lambda(C)^m=1$. 
\end{proof}

The following lemma states about properties of the quantity $r_v$, which is also used in Corollary \ref{cor:main}.

\begin{lem}\label{lem:property-r}
We will use the same notation as Definition \ref{def:r}.
Let $v$ be a vertex of $C$ in $G(f, \ZZ_{p^n})$, and assume $\bar\lambda(C)=1$.
\begin{enumerate}
\item 
We can use any representative of $v$ instead of $a_v$ 
for $r_v$ in Definition \ref{def:r}.
In particular, $f^k(a)=a+r_v p^n \pmod{p^{n+1}}$ for any representative $a$ of $v$, 
where $k$ is the size of $C$.

\item 
Either $r_v=0$ for all $v$ of $C$, or $r_v\neq 0$ for all $v$ of $C$.

\item
Assume $p>3$, or $p=3$ and $n>1$.
Let 
$v'$ be a vertex of a lifted cycle $C'$ in $G(f, \ZZ_{p^{n+1}})$ of $C$ such that $\pi_{n+1,n}(v')=v$,
where $\pi_{n+1,n}$ is the projection $\ZZ_{p^{n+1}}\to \ZZ_{p^n}$.
If $r_v \neq 0$, 
then $r_{v'} \neq 0$.

\end{enumerate}
\end{lem}

\begin{proof}
Put $g = f^k$. 

(1) Let $a, b$ be distinct representatives of $v$, and then with some $c\in [0,p)$ 
we have 
\[
b = a+c p^n \pmod{p^{n+1}}.
\]
Using the Taylor series, 
\begin{align*}
g(b)-b
&=g(a+c p^n)-(a+c p^n) \pmod{p^{n+1}}\\
&=g(a)+g'(a) c p^n -a- c p^n\\
&=g(a) -a,
\end{align*}
which implies the assertion.

(2) For another vertex $w$ of $C$ there exists an integer $l < k$ s.t. $f^l(v)=w$ 
and 
let $a$
be a representative of $v$. 
Then by (1) we have
\begin{align*}
r_{w} = \frac{g(f^l(a))-f^l(a)}{p^n}\, \%\, p 
\end{align*}
since $f^l(a)$ is a representative of $w$. 
Now 
\begin{align*}
g(f^l(a))-f^l(a) 
&=f^l(g(a))-f^l(a) \pmod{p^{n+1}}\\
&=f^l(a+r_v p^n)-f^l(a)\\
&=f^l(a)+(f^l)'(a)\,r_v p^n-f^l(a)\\
&= (f^l)'(a)\, r_v p^n.
\end{align*}
Note $(f^l)'(a)$ is not divisible by $p$, since $(f^l)'(a)$ divides $(f^k)'(a)$ and $(f^k)'(a)=1 \pmod p$.
So we have 
\[
r_v=0 \Longleftrightarrow r_{w}=0.
\]

(3) Since $r_v \neq 0$, by Theorem \ref{thm:main-0} (2) the size of $C'$ is $kp$, 
and also $\bar{\lambda}(C')=1$ by Lemma \ref{lem:property-multiplier}.
So by (1) we have 
\[
r_{v'} = \frac{f^{kp}(a')-a'}{p^{n+1}} \% p,
\]
for any representative $a'$ of $v'$.
Now by (2), to see whether $r_{v'}=0$ or not, 
it suffices to consider replacing $a'$ with a representative of any vertex of $C'$. 
Note the size of $C'$ is $kp$, thus $\pi^{-1}_{n+1,n}(v)$ is contained in $C'$, 
which implies we can choose any representative $a$ of $v$ to replace $a'$.

First, we have 
\[
f^k(a) = g(a) = a + r_v p^n + \tau p^{n+1} \pmod{p^{n+2}} %
\]
with some $\tau  \in [0,p)$, and so 
\begin{align*}
g^2(a)
&=g(a + r_v p^n + \tau p^{n+1}) \pmod{p^{n+2}}\\
&=g(a) + g'(a)(r_v p^n + \tau p^{n+1})
+ \frac{g''(a)}2 (r_v p^n + \tau p^{n+1} )^2 \\
&=a +2r_v p^n + (2\tau  +\mu) p^{n+1} + \frac{g''(a)}{2}r_v^2 p^{2n}, 
\end{align*}
where $\mu$ is such that $g'(a)=1+\mu p \pmod{p^2}$, 
and the iteration gives
\begin{align*}
g^p(a) 
&= a +p\cdot r_v p^n + \{p\, \tau  + \mu (1+2+\cdots +p-1)r_v\} p^{n+1} \\
&\qquad\qquad +\frac{g''(a)}{2}\{1+2^2+\cdots (p-1)^2\}r_v^2 p^{2n} \pmod{p^{n+2}}\\
&= a +r_v p^{n+1} +\mu \frac{(p-1)p}{2} r_v p^{n+1}+ \frac{g''(a)}{2}\frac{(p-1)p(2p-1)}{6} r_v^2 p^{2n}.
\end{align*}

If $p>3$, then integers $\frac{(p-1)p(2p-1)}{6}$ and $\frac{(p-1)p}{2}$ are divisible by $p$. 
If $p=3$ and $n>1$, then $p^{2n}=0 \pmod{p^{n+2}}$ and the integer $\frac{(p-1)p}{2}$ is divisible by $p$.
In any case the statement holds.
\end{proof}

\begin{rem}
On (3), the assumption $p>3$, or $p=3$ and $n>1$ is the best possible 
in the following sense.
If 
$p= 3, n=1$, then we have actually an example where
$r_v\neq 0$ but $r_{v'} =0$.
Indeed, let
\begin{align*}
p =3,\quad f(x) = x^2+1.
\end{align*}
Then $G(f, \ZZ_3) = \{(\bar0, \bar1), (\bar1,\bar2), (\bar2, \bar2)\}$ (represented by the directed edges), 
and there is a cycle $C=\{\bar2\}$ with 
$\bar\lambda(C)=1$. For $v=\bar2$, we have $r_{v}=\frac{f(2)-2}{3} \mathbin{\%} 3 =1$.
Now the cycle $C'=\{\bar2, \bar5, \bar8\}$ in $G(f, \ZZ_{3^2})$ is a lifted cycle of $C$. 
But, for example, for the vertex $v'=\bar2$ of $C'$, we have 
$r_{v'}=\frac{f^3(2)-2}{9} \mathbin{\%} 3 =0$.

Further if $p=2$, then $f(x)=x^3$ is an example such that 
for $n=3$ there are vertices $v, v'$ such that $r_v\neq 0$ but $r_{v'}=0$.
Indeed, let $v= \bar{3} \in \ZZ_{2^3}$, and then $C=\{v\}$ is a cycle in $G(f, \ZZ_{2^3})$ 
with $\lambda(C)=1$ and $r_v=1$. 
But concerning the lifted cycle $C'=\{\bar{3}, \bar{11}\}$ in $G(f, \ZZ_{2^4})$ of $C$, 
we have $r_{v'}=0$, where $v'=\bar{3}, \bar{11}$.
With a little more effort, we can see this phenomenon occurs for any $n\geq 3$.

\end{rem}

Using the main theorem with the lemma, 
we can give a description on the behavior of cycles on successive liftings.

\begin{cor}\label{cor:main}
Let $C$ be a cycle of size $k$ in $G(f, \ZZ_p)$. 
\begin{enumerate}[label={\upshape(\arabic*)}]
\item \label{it:cor-multi=0}
If $\bar{\lambda}(C)=0$, then for each $n>1$ there exists the only one lifted cycle in $G(f, \ZZ_{p^n})$ of $C$, which has size $k$.

\item \label{it:cor-multi!=0}
If $\bar{\lambda}(C)\neq 0$, then any vertex of the lifted graph of $C$ is part of a cycle.

\item \label{it:cor-multi=1}
Assume $p>3$ and $\bar{\lambda}(C)=1$. 
If the size of a lifted cycle $C'$ in $G(f, \ZZ_{p^N})$ of $C$ is $kp$ for some $N>1$, 
then any lifted cycle $C''$ in $G(f, \ZZ_{p^n})$ of $C'$ is of size $kp^{n-N+1}$ for $n\geq N$.

\item \label{it:cor-multi!=0,1}
If $\bar{\lambda}(C)\neq 0, 1$,
then for each $n>1$ there exists exactly one lifted cycle of size $k$ 
in $G(f, \ZZ_{p^n})$ of $C$.
Each of the other lifted cycles is of size $kmp^j$ for some $j\geq 0$,
where $m$ is the multiplicative order of $\bar{\lambda}(C)$.

\end{enumerate}
\end{cor}

\begin{proof}
(1) %
It is due to recursively using 
Theorem \ref{thm:main-0} (1) and Lemma \ref{lem:property-multiplier}. 

(2) %
By Theorem \ref{thm:main-0} (2)(3)
we see the sum of the sizes of all the lifted cycles is $kp$, 
which clearly equals the number of vertices of the lifted graph.

(3) %
Consider a sequence of cycles 
$C_1, C_{2}, C_3, \cdots$ such that
$C_{i}$ is a lifted cycle in $G(f, \ZZ_{p^{i}})$ of $C_{i-1}$, where $C_1=C$.
Let $v_i$ be a vertex in $C_i$ and put $r_i := r_{v_i}$.

Now by Theorem \ref{thm:main-0} (2)  %
we can see the size of 
$C_{n}$ 
is $k$ as long as $r_i=0$ for all $i<n$. 
So the size of $C_N$ being $kp$ means $r_{N-1} \neq 0$.  
By Lemma \ref{lem:property-r} (3) once $r_{N-1}\neq 0$, 
we must have $r_n \neq 0$ for $n>N-1$ 
and by the theorem the size of $C_n$ is $p$ times that of $C_{n-1}$ 
for $n > N-1$, which proves the statement.

(4) %
Let $G_1, G_2, G_3, \cdots$ be the sequence of the lifted graphs such that
$G_{i}$ is the lifted graph in $G(f, \ZZ_{p^{i}})$ of $G_{i-1}$, where $G_1=C$.
The proof can be done by induction as follows.

Assumption: in $G_i$ there exists exactly one cycle (denoted by $\hat{C_i}$) of size $k$ with multiplier $\bar\lambda:=\bar\lambda(\hat{C_i}) \neq 0, 1$,  
and each of the other cycles is of size $kmp^j$ for some $j\geq 0$ with multiplier $1$.

Then the lifted cycles in $G_{i+1}$ of $\hat{C_i}$ consist of one cycle of size $k$ and cycles of size $km$ 
by Theorem \ref{thm:main-0} (3). 
The former has multiplier $\bar\lambda$, 
and the latter have multiplier $1$ by Lemma \ref{lem:property-multiplier}. 
Now let $D_i$ be any other cycle than $\hat{C_i}$ in $G_i$, and then 
the lifted cycles in $G_{i+1}$ of $D_i$ clearly have size $kmp^j$ for some $j$ with multiplier $1$ 
by Theorem \ref{thm:main-0} (2). 
We have confirmed that the assumption holds for $G_{i+1}$.

In the initial case the assumption clearly holds, thus the proof completes.
\end{proof}

\begin{rem}
In the proof of (4), since another cycle than $\hat{C_i}$ has multiplier $1$, 
we can use the same argument as in the proof of (3) to give more details for $j$.
That is, assuming $p>2$, if a lifted cycle $C'$ in $G(f, \ZZ_{p^N})$ of $C$,  
is of size $kmp$ for some $N>2$, 
then any lifted cycle $C''$ in $G(f, \ZZ_{p^n})$ of $C'$ 
is of size $kmp^{n-N+1}$ for $n\geq N$.
This fact joining with (3) is a version of Theorem 2 in \cite{yoshioka_properties_2018}, 
not limited to the Chebyshev polynomials.
\end{rem}

\medskip

\bibliographystyle{plain}

\raggedright  %

\begin{thebibliography}{10}

\bibitem{bach_number_2013}
Eric Bach and Andrew Bridy.
\newblock On the number of distinct functional graphs of affine-linear transformations over finite fields.
\newblock {\em Linear Algebra Appl.}, 439(5):1312--1320, 2013.

\bibitem{bulyovszky_polynomial_2017}
Bal{\'a}zs Bulyovszky and G{\'a}bor Horv{\'a}th.
\newblock Polynomial functions over finite commutative rings.
\newblock {\em Theor. Comput. Sci.}, 703:76--86, 2017.

\bibitem{chen_polynomial_1995}
Zhibo Chen.
\newblock On polynomial functions from $\mathbb{Z}_n$ to $\mathbb{Z}_m$.
\newblock {\em Discrete Math.}, 137(1):137--145, 1995.

\bibitem{flynn_graph_2014}
Ryan Flynn and Derek Garton.
\newblock Graph components and dynamics over finite fields.
\newblock {\em Int. J. Number Theory}, 10(03):779--792, 2014.

\bibitem{gorcsos_permutation_2018}
Dalma G{\"o}rcs{\"o}s, G{\'a}bor Horv{\'a}th, and Anett M{\'e}sz{\'a}ros.
\newblock Permutation polynomials over finite rings.
\newblock {\em Finite Fields Appl.}, 49:198--211, 2018.

\bibitem{keller_counting_1968}
Gordon Keller and F.~R. Olson.
\newblock Counting polynomial functions (mod $p^n$).
\newblock {\em Duke Math. J.}, 35(4):835--838, 1968.

\bibitem{konyagin_functional_2016}
Sergei~V. Konyagin, Florian Luca, Bernard Mans, Luke Mathieson, Min Sha, and Igor~E. Shparlinski.
\newblock Functional graphs of polynomials over finite fields.
\newblock {\em J. Comb. Theory, Ser. B}, 116:87--122, 2016.

\bibitem{li_polynomial_2019}
Xiumei Li and Min Sha.
\newblock Polynomial functions in the residue class rings of {{Dedekind}} domains.
\newblock {\em Int. J. Number Theory}, 15(07):1473--1486, 2019.

\bibitem{mans_functional_2019}
Bernard Mans, Min Sha, Igor~E. Shparlinski, and Daniel Sutantyo.
\newblock On {{Functional Graphs}} of {{Quadratic Polynomials}}.
\newblock {\em Exp. Math.}, 28(3):292--300, 2019.

\bibitem{silverman_arithmetic_2007}
Joseph~H. Silverman.
\newblock {\em The {{Arithmetic}} of {{Dynamical Systems}}}, volume 241 of {\em Graduate {{Texts}} in {{Mathematics}}}.
\newblock Springer, New York, NY, 2007.

\bibitem{singh_method_2021}
Rajesh~P Singh.
\newblock A {{Method}} for {{Generating Permutation Polynomials Modulo}} $p^n$.
\newblock {\em Integers}, 21, 2021.

\bibitem{singmaster_polynomial_1974}
David Singmaster.
\newblock On polynomial functions (mod $m$).
\newblock {\em J. Number Theory}, 6(5):345--352, 1974.

\bibitem{specker_ring_2023}
Ernst Specker, Norbert Hungerb{\"u}hler, and Micha Wasem.
\newblock The ring of polyfunctions over $\mathbb{Z}/ n\mathbb{Z}$.
\newblock {\em Commun. Algebra}, 51(1):116--134, 2023.

\bibitem{yoshioka_properties_2018}
Daisaburo Yoshioka.
\newblock Properties of {{Chebyshev Polynomials Modulo}} $p^k$.
\newblock {\em IEEE Trans. Circuits Syst. Express Briefs}, 65(3):386--390, 2018.

\end{thebibliography}

\end{document}